\documentclass[12pt]{amsart}

\usepackage{amsmath,amssymb,amsfonts,amsthm,amscd,indentfirst}

\usepackage{indentfirst}
\usepackage{hyperref}
\usepackage{graphicx}

\usepackage{color}
\usepackage[dvipsnames]{xcolor}

\numberwithin{equation}{section}

\usepackage{setspace}
\usepackage{amsmath, amssymb,amsthm, amsfonts, color}

\newtheorem{thm}{Theorem}[section]
\newtheorem{lma}{Lemma}[section]
\newtheorem{prop}{Proposition}[section]

\theoremstyle{definition}
\newtheorem{definition}{Definition}[section]

\theoremstyle{remark}
\newtheorem{remark}{Remark}[section]

\newcommand{\tr}{\mbox{tr}}

\renewcommand{\div}{\mbox{div}}

\newcommand{\R}{\mathbb R}

\newcommand{\be}{\begin{equation}}
\newcommand{\ee}{\end{equation}}
\newcommand{\bee}{\begin{equation*}}
\newcommand{\eee}{\end{equation*}}
\newcommand{\bal}{\begin{aligned}}
\newcommand{\eal}{\end{aligned}}

\def\p{\partial}

\def\la{\langle}
\def\ra{\rangle}

\def\Pi{\displaystyle{\mathbb{II}}}
\renewcommand\({\left(}

\renewcommand{\o}{O(L^{n-2-2p})}

\def\vh{\vspace{.2cm}}

\def\m{\mathfrak{m}}

\def\F{\mathcal{F}}
\def\E{\mathcal{E}}

\begin{document}

\title{Mass and Riemannian Polyhedra}

\author{Pengzi Miao}
\address[Pengzi Miao]{Department of Mathematics, University of Miami, Coral Gables, FL 33146, USA}
\email{pengzim@math.miami.edu}

\author{Annachiara Piubello}
\address[Annachiara Piubello]{Department of Mathematics, University of Miami, Coral Gables, FL 33146, USA}
\email{piubello@math.miami.edu}

\thanks{P. Miao's research was partially supported by NSF grant DMS-1906423.}

\begin{abstract}
We show that the concept of the ADM mass in general relativity 
can be understood as 
the limit of the total mean curvature plus the total defect of dihedral angle
of the boundary of large Riemannian polyhedra. 
We also express the $n$-dimensional mass as a suitable integral of geometric quantities that determine 
the $(n-1)$-dimensional mass. 
\end{abstract}

\maketitle

\markboth{Pengzi Miao and Annachiara Piubello}{Mass and Riemannian Polyhedra}

\section{Introduction} 
On a complete asymptotically flat manifold with nonnegative scalar curvature, 
the Riemannian positive mass theorem asserts the mass of the manifold is nonnegative,
and is zero if the manifold is isometric to the Euclidean space. In dimension three, 
the theorem was first proved by Schoen and Yau \cite{SchoenYau79}, and by Witten \cite{Witten81}. 
Recently, a new proof was given by Bray,  Kazaras,  Khuri and Stern \cite{BKKS19}.

In \cite{Gromov14}, Gromov suggested a geometric comparison theory for scalar curvature via
the use of Riemannian polyhedra. 
More precisely, given a convex polyhedron $P^n$ in the Euclidean space $ (\R^n, \bar g)$, 
Gromov \cite{Gromov14, Gromov18a} conjectured that,  if $g$ is a Riemannian 
metric on $P$ so that
\begin{itemize}
\item  $g$ has nonnegative scalar curvature, 
\item each face of $\p P$ is weakly mean convex in $(P, g)$, and, 
\item along any edge 
of $\p P$, the dihedral angle of $(P, g)$ is no 
larger than the constant dihedral angle of $P$ in $(\R^n, \bar g)$,
\end{itemize}
then $(P, g)$ is isometric to a Euclidean polyhedron.

Significant progress toward Gromov's conjecture has been obtained by Li \cite{Li17, Li20}.
In dimension three, Li \cite{Li17} proved the conjecture for a large class of polyhedra that includes cubes and 
$3$-simplices. In higher dimensions, Li \cite{Li20} proved the conjecture for a class of 
prisms of dimensions up to $7$.

In this paper, we present a result that ties the concept of mass of asymptotically manifolds to the 
polyhedra comparison theory of Gromov.

\begin{thm} \label{thm-main}
Let $(M^n, g)$ be an asymptotically flat manifold with dimension $ n \ge 3$. 
Let $ \{ P_k \}$ denote a sequence of Euclidean polyhedra in a coordinate chart $\{ x_i \}$ 
that defines the end of $(M, g)$. Suppose $ \{ P_k \}$ satisfy the following conditions
\begin{itemize}
\item[a)]  $ P_k $ encloses the coordinate origin  and $ \lim_{k \to \infty} r_{_{P_k}} = \infty $, where
$$ r_{_{P_k} } = \min_{x \in \p P_k}  | x |  ; $$
\vh
\item[b)] $ | \F (\p P_k)  | = O ( {r^{n-1}_{_{P_k}}} )$, where $\F (\p P_k) $ denotes  the union of all the faces in $\p P_k$, 
and $ | \F ( \p P_k)  | $ is the Euclidean $ (n-1)$-dimensional volume of $ \F (\p P_k) $;
\vh
\item[c)] $ | \E  ( \p P_k)  |  = O ( r^{n-2}_{_{P_k} })$, where $ \E ( \p P_k) $ denotes the union of all the edges in $ \p P_k$, 
and  $| \E ( \p P_k)  | $ is  the Euclidean $(n-2)$-dimensional volume of $ \E ( \p P_k) $; and
\vh
\item[d)] along each edge, the Euclidean dihedral angles $\bar \alpha $ of $P_k$ satisfies 
$$  | \sin \bar \alpha |   \ge c $$ 
for some constant $ c > 0$ that is independent on $k$.
\end{itemize}
Then, as $ k \to \infty$,  the mass of $(M, g)$, which we denote by $ \m (g)$, satisfies 
\be \label{eq-main}
 \m (g)  = \frac{1}{(n-1) \omega_{n-1} } \left( - \int_{\F (\p P_k) } H \, d \sigma + \int_{\E (\p P_k) } ( \alpha - \bar \alpha) \, d \mu \right)  + o(1) . 
\ee
Here $ \omega_{n-1}$ is the volume of the standard $(n-1)$-dimensional  sphere, 
 $ H$ denotes the mean curvature of the faces $\F (\p P_k) $ in $(P_k, g)$, $\alpha$ denotes the dihedral angle along 
the edges $\E (\p P_k) $ of $(P_k, g)$, and $ d \sigma$, $ d \mu $ denote the volume element on $\F (\p P_k) $, $ \E ( \p P_k) $, respectively, with respect to the induced metric from $g$.
\end{thm}

\vspace{.1cm}

\begin{figure}[h]
\centering
\includegraphics[scale=.35]{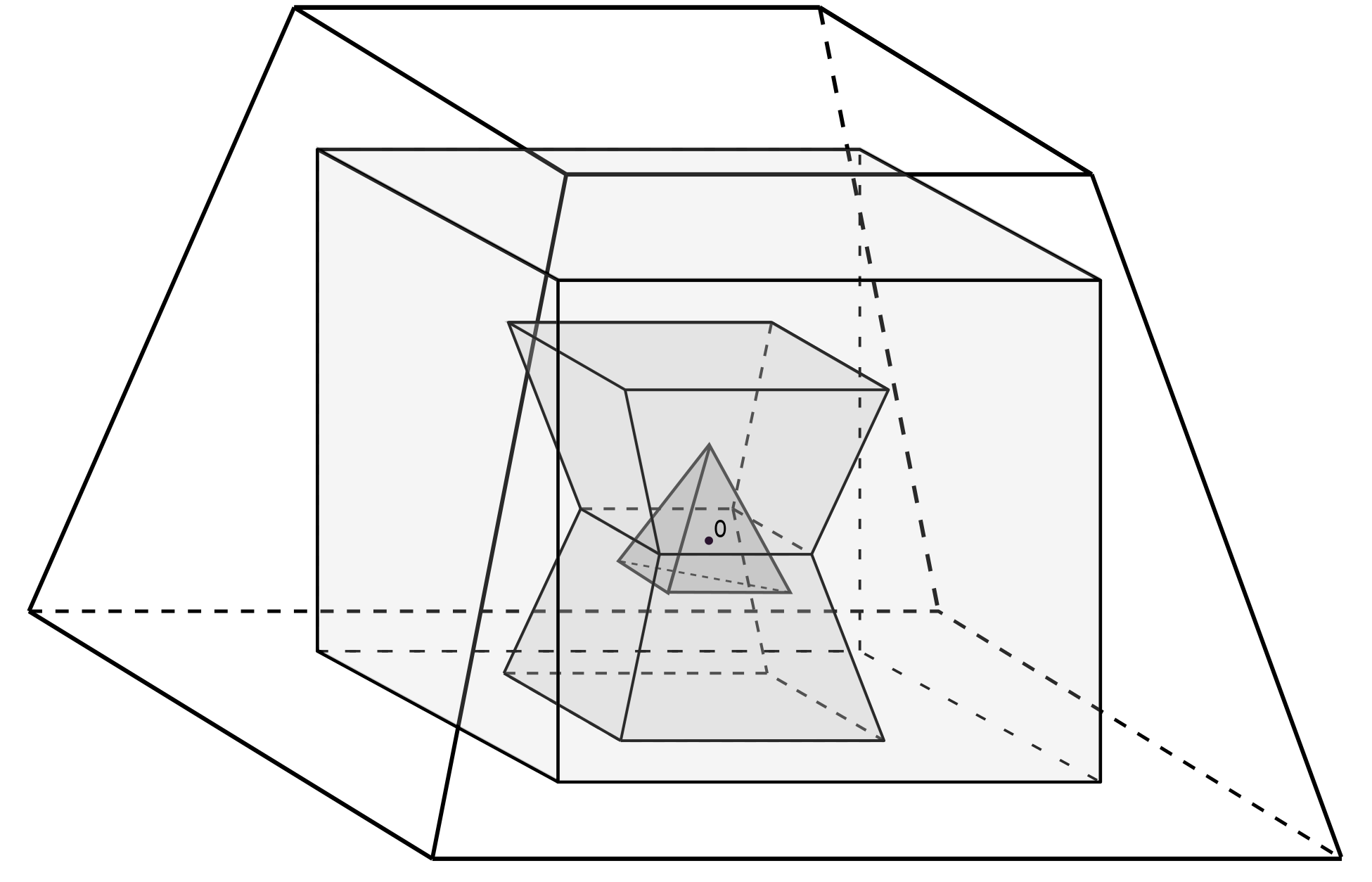}
\caption{A sequence of polyhedra $\{ P_k \}$ approaching $\infty$.}
\label{fig-p}
\end{figure}

We give some remarks about Theorem \ref{thm-main}.

\begin{remark} \label{rem-mix-nonconvexity}
$\{ P_k \}$ is allowed to consist of different type of polyhedra.
Moreover, elements in $\{ P_k \}$ can be non-convex. (See Figure \ref{fig-p}.)
\end{remark}

\begin{remark}
In $3$-dimension, if $\{ P_k \}$ is a family of coordinate cubes, \eqref{eq-main} was one of the formula derived in \cite{Miao19}. 
\end{remark}

\begin{remark}
Given any  polyhedron $P$ enclosing the coordinate origin, 
let $ P_{(r)} $ denote the scaling of $P$ by a large constant $r $ in the coordinate space. The family $\{ P_{(r)} \}$ 
satisfies all conditions a) -- d) above. As a result, formula \eqref{eq-main} holds
for any family of coordinate polyhedron $\{P_{(r)} \}$. 
\end{remark}

\begin{remark}
The angle condition $ | \sin \bar \alpha | \ge c  $ is imposed because of the formula
$ y' (x)  = - \frac{1}{\sin y} $, if $ y = \arccos x $. In the proof of Theorem \ref{thm-main}, we read 
dihedral angles from metric coefficients. A lower bound on $ | \sin \bar \alpha | $ 
serves as a sufficient condition to convey estimates on the metric to estimates on the defect of dihedral angles. 
\end{remark}

We recall the definition of an asymptotically flat manifold and its mass. 

\begin{definition} \label{def-AF}
A Riemannian manifold $(M^n, g) $ is asymptotically flat (with one end)
if there exists a compact set $K $ such that $ M \setminus K $ is diffeomorphic to 
$ \R^n \setminus B_r (0) $, where $ B_r (0)  = \{ |x| < r \} $ for some $ r > 0 $, and
with respect to the standard coordinates $\{ x_i \}$ on $\R^n$, 
the metric $g$ satisfies $ g_{ij} = \delta_{ij} + h_{ij}$, where 
\be \label{eq-g-decay}
h_{ij} =  O ( | x |^{-p} ) , \ \p h_{ij}  =  O ( | x |^{ - p - 1}) , \  
\p \p h_{ij}  =  O ( | x |^{ - p - 2 } )
\ee 
for some $ p > \frac{n-2}{2} $.
Moreover, the scalar curvature of $g$ is integrable on $(M, g)$.
\end{definition}

On an asymptotically flat manifold $ (M^n, g) $, the ADM mass \cite{ADM} is given by 
\be \label{eq-mass-def}
\m (g) = \lim_{ r \rightarrow \infty} \frac{1}{ 2 (n-1) \omega_{n-1} } \int_{ S_r }  (g_{ij,i} - g_{ii,j} ) \, \bar \nu^j \, d \bar \sigma .
\ee
Here $ S_r  = \{ x  \ | \ | x | = r \}$ denotes a coordinate sphere, 
$ \bar \nu $ is the Euclidean  outward unit normal to $ S_r$, $ d \bar \sigma$ is the Euclidean 
volume element on $ S_r$, and summation is applied over repeated indices. 
It was shown by Bartnik \cite{Bartnik86}, and by Chru\'sciel \cite{Chrusciel} independently,  that $\m(g)$
is a geoemtric invariant of $(M, g)$,  independent on coordinates satisfying \eqref{eq-g-decay}.

Often it is convenient to compute the mass  in \eqref{eq-mass-def} with
$ \{ S_r \}$ replaced by $\{ \p D_k\}$, where $\{ D_k \}$ is another suitable exhaustion sequence of $(M, g)$.
By Proposition 4.1 in \cite{Bartnik86} (more specifically, by its proof), it is known  
\be  \label{eq-mass-def-p}
\m (g) = \lim_{ k \rightarrow \infty} \frac{1}{ 2 (n-1) \omega_{n-1} } \int_{ \p P_k }  (g_{ij,i} - g_{ii,j} ) \, \bar \nu^j \, d \bar \sigma .
\ee
Here $\{ P_k \}$ is a sequence of polyhedra satisfying conditions a) and b) in Theorem \ref{thm-main}.

The positive mass theorem in general dimensions was shown by Schoen-Yau \cite{SchoenYau17}.
(Also see the work of Lohkamp \cite{Lo3}).
The following is a corollary of the positive mass theorem and Formula \eqref{thm-main}.

\begin{thm} \label{thm-cor}
Let $(M^n, g)$ be a complete, asymptotically flat manifold with nonnegative scalar curvature.  
Let $ \{ P_k \}$ denote a sequence of Euclidean polyhedra satisfying conditions in Theorem \ref{thm-main}.
Then, for large $k$,
\be \label{eq-main-cor}
- \int_{\F (\p P_k) } H \, d \sigma + \int_{\E (\p P_k) } ( \alpha - \bar \alpha) \, d \mu \ge 0 .
\ee
In particular, for any fixed Euclidean polyhedron  $ P$, 
\be
- \int_{\F_{(r) }} H \, d \sigma + \int_{\E_{(r) }} ( \alpha - \bar \alpha) \, d \mu \ge 0 ,
\ee
where $\F_{(r)}$ and $ \E_{(r)}$ are the faces and edges of the polyhedron $P_{(r)}$ obtained by 
scaling $P$ by a large constant factor $r$.
\end{thm}

The polyhedra $P_{k}$, $P $ in Theorem \ref{thm-cor} do not need to be convex (see Remark \ref{rem-mix-nonconvexity}). 
In this sense, Theorem \ref{thm-cor} indicates that
Gromov's comparison theory of scalar curvature might extend to non-convex Euclidean polyhedra.

In \cite{Miao19}, $\{ P_k \}$ was chosen as a family of coordinate cubes in $3$-dimension, 
and a formula that represents the $3$-dimensional 
mass via suitable integration of the angle defect, measured by the boundary term in the Gauss-Bonnet theorem, 
was obtained. 
In higher dimensions, by taking $\{ P_k \}$ as coordinate cubes, 
we give an induction-type formula
that evaluates the $n$-dimensional mass via quantities that determine the $(n-1)$-dimensional mass. 

\begin{thm}\label{thm-prop-intro}
Let $(M^n, g)$ be an asymptotically flat manifold with dimension $ n \ge 4$.
Given any $k \in \{ 1, \cdots, n \}$,  any large constant $L$, and any $ t \in [-L, L]$,  
there is a quantity $ \m_k^{(n-1)} (t, L)$, associated to the coordinate hyperplane 
$\{ x_k = t \}$, such that 
\be
\m ( g|_{  \{ x_k = t \} } ) = \lim_{ L \to \infty } \m_k^{(n-1)} (t, L) ,
\ee
and
\begin{equation}\label{eq-slice-intro}
  \m (g)=  \frac{ \omega_{n-2} }{ (n-1) \omega_{n-1} }   \lim_{ L \to \infty} \sum_{k=1}^n \int_{-L}^L \m^{(n-1)}_k ( t, L ) \, dt .
\end{equation}
\end{thm}

We provide the precise definition of  $ \m_k^{(n-1)} (t, L)$ and the proof of Theorem \ref{thm-prop-intro} 
in Section \ref{sec-cube}. 

\section{Difference of the mean curvature}

In this section, we derive a formula on the difference of the mean curvatures of a hypersurface 
with respect to two metrics that are close. The formula will be used in the next section to prove
Theorem \ref{thm-main}.

\vh

\noindent {\bf Setting}:  Let $ \Sigma^{n-1}$ be a hypersurface in a manifold $U^n$. Let $ \bar g$, $g$ be two Riemannian 
metrics on $U$. We view $\bar g$ as a background metric. 
Let $\bar \gamma$, $ \gamma$ be the metrics on $ \Sigma$ induced from $\bar g$, $ g$, respectively. 
Let $ \bar \nu$, $ \nu$ denote the unit normal vectors to $ \Sigma$ in $(M, \bar g)$, $(M, g)$, respectively,  so that they 
point to the same side of $\Sigma$. 
Let $ \bar A$, $ \bar H $ and $ A$, $ H$ be the second fundamental form, the mean curvature of $\Sigma$
with respect to $ \bar \nu$ and $ \nu$, in $(M, \bar g)$ and $(M, g)$, respectively.  

\vh 

\noindent {\bf Assumption}:  We write $ g = \bar g + h $, and assume 
\be \label{eq-condition-small}
 | h |_{\bar g} < \epsilon (n)  \ \text{in}  \ U . 
\ee
Here $ \epsilon (n)  < \frac{1}{n-1}$ is a positive constant that depends only on $n$, and 
$ | \cdot |_{\bar g}$ denotes the norm of a tensor with respect to $\bar g$.

\vspace{.1cm}

\begin{prop} \label{prop-mean-curvature}
The difference of the two mean curvatures satisfies 
\be \label{eq-H-bar-H-prop}
\begin{split}
2 ( H - \bar H ) 
= & \ ( d \, \tr_{\bar g} \, h - \div_{\bar g} \, h ) (\bar \nu) - \div_{\bar \gamma} X - \la h, \bar A \ra_{\bar \gamma}  \\
& \  + | \bar A |_{\bar g} \, O ( | h |^2_{\bar g} ) + O ( | \bar D h |_{\bar g} \, | h |_{\bar g} ).
\end{split}
\ee
Here $ \div_{\bar g} (\cdot)$ and $ \tr_{\bar g} (\cdot) $ denote the divergence and trace on $(U, \bar g)$;
$ \div_{\bar \gamma} (\cdot) $ denotes the divergence on $(\Sigma, \bar \gamma)$;
 $X $ is the vector field on $ \Sigma$ that is dual to the $1$-form $h(\bar \nu, \cdot) $ with respect to  $\bar \gamma$. 
Given two functions $f$ and $\phi$, we write $ f =  O ( \phi )$ to denote $ | f  | \le C | \phi | $ with a constant $ C $ that depends 
only on $n$.
\end{prop}

We remark that a linearized version of \eqref{eq-H-bar-H-prop} can be found  in \cite[Equation (34)]{MT09}.
Below we verify \eqref{eq-H-bar-H-prop}. 
Given any $ p \in \Sigma$, let $\{ x_i \}_{1\le i \le n} $ be local coordinates around $p$ in $M$ 
such that $ \{ x_\alpha \}_{ 1 \le \alpha \le n-1}$ are coordinates near $p$ on $\Sigma$.
By definition, 
\be \label{eq-H-bar-H-0}
\begin{split}
H - \bar H = & \ \gamma^{\alpha \beta} A_{\alpha \beta} - \bar \gamma^{\alpha \beta} \bar A_{\alpha \beta} \\
= & \ \bar \gamma^{\alpha \beta} ( A_{\alpha \beta} - \bar A_{\alpha \beta} ) 
+  (  \gamma^{\alpha \beta} - \bar \gamma^{\alpha \beta} ) \bar A_{\alpha  \beta}
+  (  \gamma^{\alpha \beta} - \bar \gamma^{\alpha \beta} ) ( A_{\alpha  \beta} - \bar A_{\alpha \beta} ) ,
\end{split}
\ee
where
$ A_{\alpha \beta} = - \la D_{\p_\alpha} \p_\beta ,  \nu \ra_{g} $ and 
$  \bar A_{\alpha \beta}  =  - \la \bar D_{\p_\alpha} \p_\beta , \bar \nu \ra_{\bar g} $.

We estimate $ \gamma^{\alpha \beta} - \bar \gamma^{\alpha \beta}$, $ g^{ij} - \bar g^{ij}$, 
and $ \nu - \bar \nu$ first.
For convenience, suppose $\{ \p_i \}$ are orthonormal with respect to $\bar g$ at $ p$.
In particular, $ \p_n = \bar \nu$ at $p$.

\begin{lma} \label{lem-0-order-geometry}
In the above coordinates, 
\be
\gamma^{\alpha \delta} = \bar \gamma^{\alpha \delta} - h_{\alpha \delta}  + O ( | h |_{\bar \gamma}^2 ) ,
\ \ 
g^{ij} = \bar g^{ij} - h_{ij} + O ( | h |_{\bar g}^2 ) ,
\ee
\be
\nu - \bar \nu = \left( - h_{n \beta} + O (|h|_{\bar g}^2  ) \right) \p_\beta +   (  - \frac12 h_{nn} + O ( |h |^2_{\bar g} ) ) \p_n .
\ee
Here repeated indices denote summation over those indices. 
In terms of $X$, this gives
\be
\left| \nu - \bar \nu  + X +  \frac12 h (\bar \nu, \bar \nu) \bar \nu \right|_{\bar g} = O (|h|_{\bar g}^2  ) .
\ee 
\end{lma}

\begin{proof}
Let $ \tau $ denote the tangential restriction of $h$ to $\Sigma$.
At $p$,  $ \gamma_{\alpha \beta} = \delta_{\alpha \beta} + \tau_{\alpha \beta} $. 
Write 
$ \gamma^{\alpha \beta} = \delta^{\alpha \beta} - \eta^{\alpha \beta} $, 
then
\be \label{eq-tau-eta-1}
\tau_{\alpha \delta} - \tau_{\alpha \beta} \eta^{\beta \delta} = \eta^{\alpha \delta}.
\ee
This and \eqref{eq-condition-small} imply 
$ \epsilon +  (n-1)  M \epsilon \ge | \eta^{\alpha \delta} | $,
where $ M = \max_{ \alpha , \beta } \,  | \eta^{\alpha \beta} | $.
As $\alpha$, $\delta$ are arbitrary,  $ \epsilon +  (n-1)  M \epsilon \ge  M  $. 
As a result, 
$  | \eta^{\alpha \delta} | \le C_1 $. 
Here and below, $\{ C_i \} $ denote positive constants that only depend on $n$.
By \eqref{eq-tau-eta-1}, 
\be
 | \eta^{\alpha \delta} | \le C_2  \, \max_{\alpha , \delta} | \tau_{\alpha \delta} | .
\ee
This in turn implies
\be
\gamma^{\alpha \delta} = \bar \gamma^{\alpha \delta} - \tau_{\alpha \delta}  + O ( | \tau |_{\bar \gamma}^2 ) . 
\ee
For the same reason,
\be
g^{ij} = \bar g^{ij} - h_{ij} + O ( | h |_{\bar g}^2 ) .
\ee

Next, we consider $ \nu - \bar \nu$, where $\bar \nu = \p_n$. 
The condition $\la \nu, \nu \ra_g = 1 $ gives
\be \label{eq-nu-0}
 \nu^i \nu^j g_{ij}  = \nu^i \nu^j ( \delta_{ij} + h_{ij} ) = 1, 
\ee
which implies $ \nu^i = O (1) $. 
The condition $ \la \nu, \p_\beta \ra_g = 0$ gives
\be \label{eq-nu-2}
0 = \nu^\alpha g_{\alpha \beta} + \nu^n g_{n \beta} 
= \nu^\beta + \nu^\alpha  h_{\alpha \beta}   + \nu^n h_{n \beta} .
\ee
Combined with $ \nu^i = O (1) $, this gives
$  \nu^\beta = O ( |h|_{\bar g} ) $.
Coming back to \eqref{eq-nu-0}, we have 
\be \label{eq-nu-n-2}
1 = \nu^\alpha \nu^\beta g_{\alpha \beta} + 2 \nu^\alpha \nu^n g_{\alpha n} + ( \nu^n )^2 g_{nn}  .
\ee
Thus,  $ 1 = O ( |h |^2_{\bar g} ) +  (\nu^n)^2  ( 1 + h_{nn} ) $. 
This shows
\be \label{eq-nu-n-1}
\nu^n = 1 - \frac12 h_{nn} + O ( |h |^2_{\bar g} ) .
\ee
Plugging \eqref{eq-nu-n-1} into \eqref{eq-nu-2}, we have
\be \label{eq-nu-3}
\nu^\beta + \nu^\alpha  h_{\alpha \beta}   + ( 1 + O (|h|_{\bar g} )  h_{n \beta} = 0 ,
\ee
which gives  
$ \nu^\beta = - h_{n \beta} + O (|h|_{\bar g}^2  ) $. 
This proves the lemma.
\end{proof}

To estimate the difference of $ A $ and $ \bar A$, 
 let $ D$ and $ \bar D$ denote the connection of  $g$ and $ \bar g$, respectively. 
Let  $ T = D - \bar D$, 
then $ T$ is a $(1,2)$ tensor. 
With respect to  $\{ x_i \}$, if we write 
$ D_{\p_i} \p_j - \bar D_{\p_i} \p_j = T_{ij}^k \, \p_k $, 
then 
\be
T_{ij}^k = \frac12 g^{kl} ( g_{lj; i} + g_{il;j} - g_{ij;l} ).
\ee
Here ``$ ; $" denotes covariant differentiation with respect to $\bar g$. 
By Lemma \ref{lem-0-order-geometry}, 
\be \label{eq-T-ijk}
\begin{split}
T_{ij}^k = & \ \frac12 g^{kl} ( h_{lj; i} + h_{il;j} - h_{ij;l} ) \\
 = & \  \frac12 ( \bar g^{kl} - h_{kl} + O ( | h |_{\bar g}^2) )  ( h_{lj; i} + h_{il;j} - h_{ij;l} ) \\
 = & \  \frac12 ( h_{kj; i} + h_{ik;j} -   h_{ij;k}  )  - \frac12  h_{kl} ( h_{lj; i} + h_{il;j} - h_{ij;l} ) + O ( |h|_{\bar g}^2 \, | \bar D h |_{\bar g} ) .
\end{split}
\ee

\begin{lma} \label{lem-II-forms}
The second fundamental forms $A$ and $ \bar A$ satisfy 
\be \label{eq-A-barA-d}
A_{\alpha \beta} - \bar A_{\alpha \beta} 
= \frac12 \bar A_{\alpha \beta} h_{nn} -  T_{\alpha \beta}^n  +  \bar A_{\alpha \beta} \, O ( | h |^2_{\bar g} )   
+ O ( | \bar D h |_{\bar g} \, | h |_{\bar g} ) .
\ee
\end{lma}

\begin{proof} 
By definition, 
\be
\begin{split}
A_{\alpha \beta} - \bar A_{\alpha \beta} = & \ \la \bar D_{\p_\alpha} \p_\beta, \bar \nu \ra_{\bar g} - \la D_{\p_\alpha} \p_\beta, \nu \ra_g  \\ 
= & \ ( \bar D_{\p_\alpha} \p_\beta )^i \bar \nu^j \bar g_{ij}  
- [ ( \bar D_{\p_\alpha} \p_\beta )^i  + T_{\alpha \beta}^i ]( \bar \nu^j +  \nu^j  - \bar \nu^j ) ( \bar g_{ij} +  h_{ij} ) .
\end{split}
\ee
For convenience, we can make the following simplification. 
In addition to assuming $ \{ \p_i \}$ are orthonormal at $p $ with respect to $\bar g$, 
we also assume that  $\{ x_i \}$ are chosen so that their restrictions to $\Sigma$, 
$\{ x_\alpha \}$, are normal at $ p $ on $(\Sigma, \bar \gamma)$. As a result, at $ p $,  
\be \label{eq-simplification}
( \bar D_{\p_\alpha } \p_\beta )^\delta = 0 , \ \text{and} \  ( \bar D_{\p_\alpha } \p_\beta)^n = - \bar A_{\alpha \beta} . 
\ee

By Lemma \ref{lem-0-order-geometry} and \eqref{eq-simplification}, we  have
\be \label{eq-start-1}
\begin{split}
- ( \bar D_{\p_\alpha} \p_\beta )^i  (  \nu^j  - \bar \nu^j ) \bar g_{ij} 
= & \ - ( \bar D_{\p_\alpha} \p_\beta )^n  (  \nu^n  - \bar \nu^n )  \\
= & \ \bar A_{\alpha \beta}  ( - \frac12 h_{nn} ) +  \bar A_{\alpha \beta} \, O ( | h |^2_{\bar g} )   ,
\end{split}
\ee
\be
\begin{split}
 -  T_{\alpha \beta}^i ( \bar \nu^j +  \nu^j  - \bar \nu^j ) \bar g_{ij} 
= & \ -  T_{\alpha \beta}^n  -  T_{\alpha \beta}^i (  \nu^j  - \bar \nu^j ) \bar g_{ij} ,
\end{split} 
\ee
\be
- ( \bar D_{\p_\alpha} \p_\beta )^i  ( \bar \nu^j  )  h_{ij} = \bar A_{\alpha \beta} \, h_{nn} ,
\ee
\be
\begin{split}
 - ( \bar D_{\p_\alpha} \p_\beta )^i  ( \nu^j  - \bar \nu^j  )  h_{ij} 
= & \   \bar A_{\alpha \beta} ( \nu^\delta  - \bar \nu^\delta  )  h_{n \delta }  + \bar A_{\alpha \beta} ( \nu^n  - \bar \nu^n  )  h_{n n } \\
= & \ \bar A_{\alpha \beta} ( - h_{n \delta}  )  h_{n \delta }  + \bar A_{\alpha \beta} ( - \frac12 h_{nn} )  h_{n n } 
+  \bar A_{\alpha \beta} \, O ( | h |_{\bar g}^3 ) ,
\end{split}
\ee
\be \label{eq-end-1}
 T_{\alpha \beta}^i ( \bar \nu^j +  \nu^j  - \bar \nu^j )  h_{ij} 
 = T_{\alpha \beta}^i \bar \nu^j h_{ij}  + O ( | \bar D h |_{\bar g} \, | h |_{\bar g}^2 ) . 
\ee
Equation \eqref{eq-A-barA-d}  follows from \eqref{eq-start-1} -- \eqref{eq-end-1}.
\end{proof}

\vspace{.1cm}

\begin{proof}[Proof of Proposition \ref{prop-mean-curvature}]
We recall from \eqref{eq-H-bar-H-0} that
\be \label{eq-H-bar-H-1}
\begin{split}
H - \bar H 
= & \ \bar \gamma^{\alpha \beta} ( A_{\alpha \beta} - \bar A_{\alpha \beta} ) 
+  (  \gamma^{\alpha \beta} - \bar \gamma^{\alpha \beta} ) \bar A_{\alpha  \beta}
+  (  \gamma^{\alpha \beta} - \bar \gamma^{\alpha \beta} ) ( A_{\alpha  \beta} - \bar A_{\alpha \beta} ) .
\end{split}
\ee
By Lemma \ref{lem-II-forms}, 
\be \label{eq-H-barH-0}
\bar \gamma^{\alpha \beta} ( A_{\alpha \beta} - \bar A_{\alpha \beta} ) 
=   \frac12 \bar H  h_{nn} - \bar \gamma^{\alpha \beta}  T_{\alpha \beta}^n  + \bar H \, O ( | h |^2_{\bar g} )   .
+ O ( | \bar D h |_{\bar g} \, | h |_{\bar g} ) ,
\ee
where,  by \eqref{eq-T-ijk}, 
\be
\begin{split}
 - \bar \gamma^{\alpha \beta}  T_{\alpha \beta}^n 
 =  & \ - \bar \gamma^{\alpha \beta}  \frac12 ( h_{n  \beta; \alpha } + h_{\alpha n; \beta} -   h_{\alpha \beta; n }  )  + O ( |h|_{\bar g} \, | \bar D h |_{\bar g} ) \\
 =  & \ - (\div_{\bar g} h )_{n}  +  \frac12 (d \, \tr_{\bar g} h )_{n} + \frac12 h_{nn;n} + O ( |h|_{\bar g} \, | \bar D h |_{\bar g} ) .
\end{split}
\ee
By Lemma \ref{lem-0-order-geometry}, 
\be
(  \gamma^{\alpha \beta} - \bar \gamma^{\alpha \beta} ) \bar A_{\alpha  \beta}
=  - \la  h, \bar A \ra_{\bar \gamma} + O ( | h |_{\bar \gamma}^2 ) \, | \bar A|_{ \bar \gamma} .
\ee
Moreover, by Lemma \ref{lem-0-order-geometry} and Lemma \ref{lem-II-forms}, 
\be \label{eq-3rd-term}
 (  \gamma^{\alpha \beta} - \bar \gamma^{\alpha \beta} ) ( A_{\alpha  \beta} - \bar A_{\alpha \beta} ) 
 =  O ( | h |^2_{\bar g} ) \, | \bar A |_{\bar g}  + O ( | \bar D h |_{\bar g} \, | h |_{\bar g} ).
\ee
Therefore, it follows from \eqref{eq-H-bar-H-1} -- \eqref{eq-3rd-term} that
\be \label{eq-H-M03}
\begin{split}
 H - \bar H 
= & \ \frac12 \bar H  h_{nn} +  \frac12 h_{nn;n} - \left( \div_{\bar g} h  - \frac12 d \, \tr_{\bar g} h \right)_{n}   - \la h , \bar A \ra_{\bar \gamma}  \\    
& \  + | \bar A |_{\bar g} \, O ( | h |^2_{\bar g} ) + O ( | \bar D h |_{\bar g} \, | h |_{\bar g} ).
\end{split}
\ee
(A linearized version of \eqref{eq-H-M03} can be found in \cite[Equation 42]{Miao02}.)

On the other hand, we have an identity 
\be \label{eq-identity-MT}
(\div_{\bar g} h)_n = h_{nn;n} + \div_{\bar \gamma } X + \bar H h_{nn} - \la h, \bar A \ra_{\bar \gamma} .
\ee
(See  \cite[Equation (32)]{MT09} for instance.)
\eqref{eq-H-bar-H-prop} now follows from \eqref{eq-H-M03} and \eqref{eq-identity-MT}.
\end{proof}

\section{Mass flux across the boundary of a polyhedron}

Let $(M^n, g)$ be an asymptotically flat manifold. Let $ K \subset M $ and $ B_r (0) \subset \R^n $ 
be given in Definition \ref{def-AF}.
In what follows, we identify $ M \setminus K $ with $ \R^n \setminus B_r (0)$, and let 
$ \bar g = \delta_{ij} d x_i \, d x_j $ denote the background Euclidean metric.

We consider a polyhedron $P \subset  (\R^n, \bar g) $ such that  the interior of $P$ contains $ B_r (0)$.
Let $ \p P$ be the boundary of $ P$, which is a union of finitely many faces 
$$ F_1, \cdots, F_{ f (P)} , $$
where $ f(P)$ denotes the number of faces in $ \p P$. 
Each face $ F_{_A}$, $ 1 \le A \le f(P)$, is an $(n-1)$-dimensional polyhedron lying in a hyperplane  
in $ (\R^n, \bar g)$.
We let
$$ r_{_P} = \min_{ x \in \p P } | x | .$$
As $g$ is asymptotically flat, we assume $ r_{_P}$ is sufficiently large so that
\be
| h |_{\bar g} (x) < \epsilon (n) ,   \ \ \text{if} \ |x| > \frac12 r_{_P}. 
\ee
Here $ h = g - \bar g $ and $ \epsilon (n) $ is the constant in \eqref{eq-condition-small}.

Let $ F$ be a face of $ \p P$.  
Let $ d \sigma$, $ d \bar \sigma$ denote the volume measure on $F$ with respect to the induced metrics
$\gamma$, $ \bar \gamma$, respectively. 
Let $ \nu$, $ \bar \nu$ be the outward unit normal to $ F$ in 
$(M, g)$, $(\R^n, \bar g)$, respectively. Let $H$ be the mean curvature of $F$ in $(M, g)$
with respect to $\nu$.
By Proposition \ref{prop-mean-curvature}, 
\be \label{eq-int-H}
\begin{split}
2 \int_{F} H \, d \bar \sigma = 
\int_F ( h_{ii,j} - h_{ij,i} ) \bar \nu^j \, d \bar \sigma - \int_{F} \div_{\bar \gamma} X \, d \bar \sigma
+ \int_{F} O ( | \bar D h |_{\bar g} \, |h |_{\bar g} ) \, d \bar \sigma . 
\end{split}
\ee
Here we made use of the fact that $ F$ is totally geodesic in $(\R^n, \bar g)$.

By \eqref{eq-H-M03}, $ H = O ( r_{_P}^{-p - 1} )$. By \eqref{eq-g-decay}, $ d \sigma = ( 1 + O ( r_{_P}^{-p} ) ) \, d \bar \sigma $.
Thus,
\be \label{eq-int-H-1}
\int_F H \, d \bar \sigma =  \int_F H \, d \sigma + |F |_{\bar \gamma }  \, O ( r_{_P}^{- 2 p - 1} ) . 
\ee
Here $ | F |_{\bar \gamma} $ denotes the $(n-1)$-dimensional volume of $F$ in $(\R^n, \bar g)$.

Similarly, by \eqref{eq-g-decay}, 
\be \label{eq-dh-h}
\int_{F} O ( | \bar D h |_{\bar g} \, |h |_{\bar g} ) \, d \bar \sigma = | F |_{\bar \gamma} \, O ( r_{_P}^{-2 p - 1} ) .
\ee

Integrating by parts, we have
\be \label{eq-F-2-E}
\int_F \div_{\bar \gamma} X \, d \bar \sigma = \int_{\p F} \la X, \bar n \ra_{\bar \gamma} \, d \bar \mu .
\ee
Here $\bar n $ denotes the outward unit normal to $\p F$ in $(F, \bar \gamma)$, and $d \bar \mu$ is the 
volume element on $ \p F$ induced from $\bar \gamma$. By the definition of $X$ and $h$, 
\be
\la X, \bar n \ra_{\bar \gamma} = h (\bar \nu, \bar n ) = g (\bar \nu, \bar n ). 
\ee

Next we consider two adjacent faces $F_{_A}$ and $F_{_B}$. Let $ \theta $ be the angle between $\nu_{_A} $ and $\nu_{_B}$
in $(M, g)$, and let $\bar \theta $ be the angle between $\bar \nu_{_A}$ and $\bar \nu_{_B}$ in $ (\R^n, \bar g)$. 

The contribution from 
$$ \int_{F_{_A} } \div_{\bar \gamma} X \, d \bar \sigma \ \ \text{and} \ \  
 \int_{F_{_B} } \div_{\bar \gamma} X \, d \bar \sigma $$
on the edge $F_{_A} \cap F_{_B}$ is 
\be \label{eq-edge-contribution}
\int_{ F_{_A} \cap F_{_B} } g (\bar \nu_{_A}, \bar n_{_A} ) + g (\bar \nu_{_B}, \bar n_{_B}  ) \, d \bar \mu . 
\ee

In what follows, we write
$$ \bar \nu_{_A} = (a_1, \cdots, a_n ) \  \ \text{and} \ \  \bar \nu_{_B} = (b_1, \cdots, b_n ) , $$
where
$ \sum_{i=1}^n a_i^2 = 1 $ and $ \sum_{i=1}^n b_i^2 = 1 $.
Using the fact $ F_{_A}$ and $ F_{_B}$ are, respectively, part of a level set of the functions
$$ f_{_A} (x) =  a_i x_i  \ \ \text{and} \ \ f_{_B} (x) = b_i x_i , $$
we know that the outward unit normal $\nu_{_A}$, $ \nu_{_B}$ to $F_{_A}$, $ F_{_B}$, respectively
with respect to $g$, is given by 
\be
\nu_{_A} = \frac{ a^i e_i}{ ( a_i a_j g^{ij}  )^\frac12 } \ \
\text{and} \ \
\nu_{_B} = \frac{ b^i e_i}{ ( b_i b_j g^{ij}  )^\frac12 }.
\ee
Here $ a^i = g^{ij} a_j $,  $ b^i = g^{ik} b_k $, 
and $e_i  = \p_{x_i} $.

By definition, 
\be
\begin{split}
\cos \theta =   g( \nu_{_A}, \nu_{_B} )
= \frac{ a_i b_j g^{ij} }{ ( a_i a_j g^{ij} )^\frac12 ( b_i b_j g^{ij}  )^\frac12  }.
\end{split}
\ee
By Lemma \ref{lem-0-order-geometry},
\be
a_i b_j g^{ij} = \bar g ( \bar \nu_{_A}, \bar \nu_{_B} )  - a_i b_j h_{ij} + O ( r_{_P}^{-2p} ),
\ee
\be
( a_i a_j g^{ij}  )^\frac12 =  1 - \frac12  a_i a_j h_{ij} + O ( r_{_P}^{-2p} ) ,
\ee
\be
\begin{split}
( b_i b_j g^{ij}  )^\frac12 =  1 - \frac12  b_i b_j h_{ij} + O ( r_{_P}^{-2p} ) .
\end{split}
\ee
Therefore,
\be
\begin{split}
\cos \theta = & \  \left[ \bar g (  \bar \nu_{_A}, \bar \nu_{_B} )  - a_i b_j h_{ij}  + O ( r_{_P}^{-2p} ) \right]
\left[ 1 + \frac12  a_i a_j h_{ij} + \frac12  b_i b_j h_{ij}  + O ( r_{_P}^{-2p} ) \right] \\
= & \   \bar g ( \bar \nu_{_A}, \bar \nu_{_B} )  
\left[ 1 + \frac12  a_i a_j h_{ij} + \frac12  b_i b_j h_{ij} \right]  - a_i b_j h_{ij}  + O ( r_{_P}^{-2p} ) .
\end{split}
\ee
Since $ \cos \bar \theta =  \bar g ( \bar \nu_{_A}, \bar \nu_{_B} )$, this gives 
\be \label{eq-E}
\begin{split}
\cos \theta - \cos \bar \theta 
= & \ \frac12   \cos \bar \theta \left(  a_i a_j h_{ij} +   b_i b_j h_{ij} \right)  - a_i b_j h_{ij}  + O ( r_{_P}^{-2p} ) .
\end{split}
\ee

Next we consider two cases depending on whether $ P$ is convex at the edge $F_{_A} \cap F_{_B} $.

\vh

\noindent {\bf Case 1}. $ P$ is convex at $F_{_A} \cap F_{_B} $. 
This means, if $ \bar \alpha$ is the Euclidean dihedral angle of $P$ at $ F_{_A} \cap F_{_B}$, 
then $ 0 < \bar \alpha < \pi $.

\begin{figure}[h]
\centering
\includegraphics[scale=1]{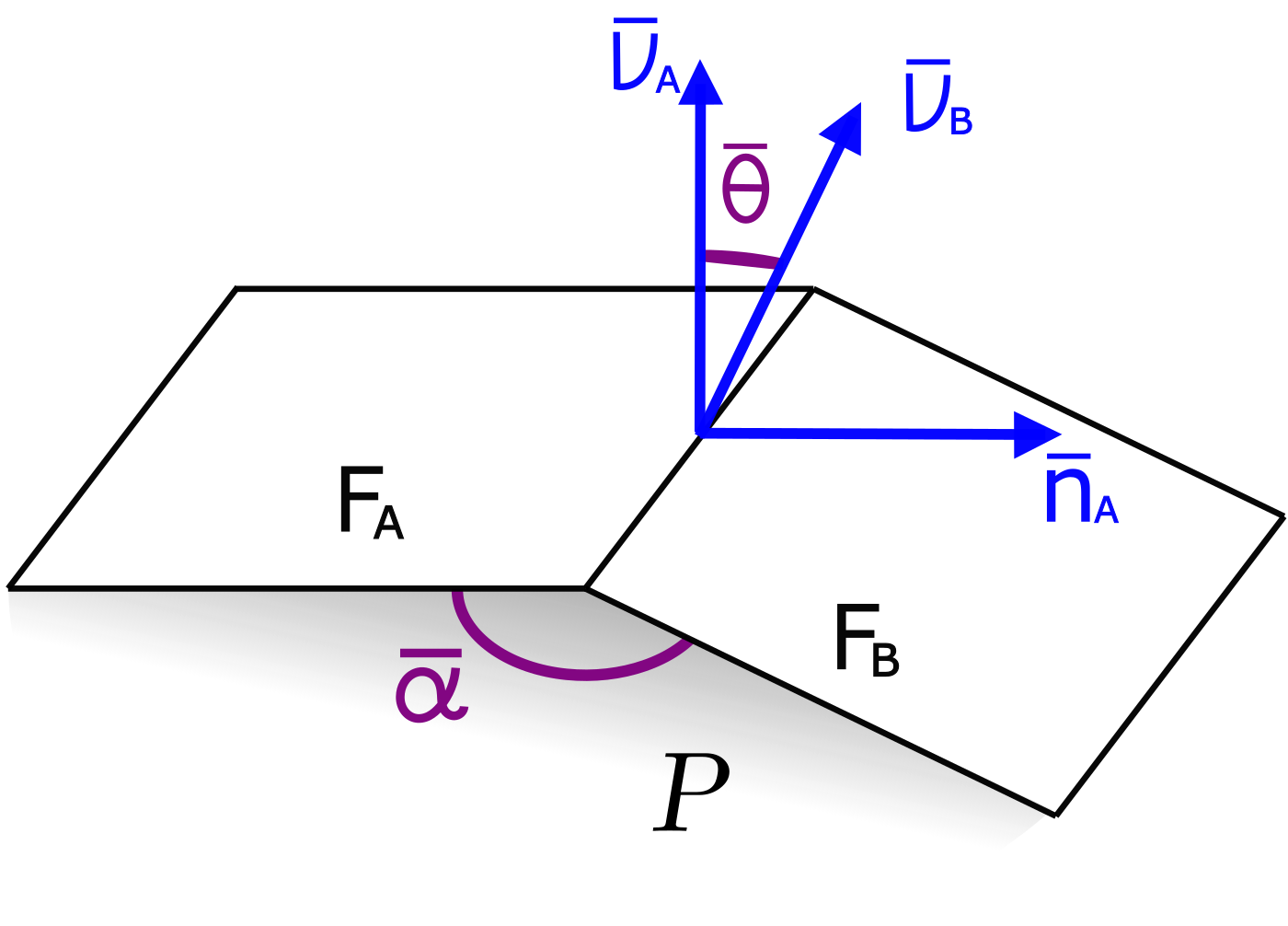}
\caption{$P$ is convex at $ F_{_A} \cap F_{_B}$.}
\label{fig-convex}
\end{figure}

In this case, we have 
\be \label{eq-alpha-theta-convex}
\bar \alpha + \bar \theta = \pi  \ \ \text{and} \ \
(\sin \bar \theta) \, \bar n_{_A} = - (\cos \bar \theta) \,  \bar \nu_{_A} + \bar \nu_{_B} .
\ee
Hence, 
\be \label{eq-f-a}
\begin{split}
\sin \bar \theta \,  g (\bar \nu_{_A} , \bar n_{_A} ) 
= & \ g ( \bar \nu_{_A} , - (\cos \bar \theta) \,  \bar \nu_{_A} + \bar \nu_{_B} ) \\
= & \ - \cos \bar \theta \, ( 1 + a_i a_j h_{ij}  ) + \cos \bar \theta  + a_i b_j h_{ij}  \\
= & \  \frac12   \cos \bar \theta \left( -  a_i a_j h_{ij} +   b_i b_j h_{ij} \right)  + \cos \bar \theta -  \cos \theta 
+ O ( r_{_P}^{-2p} ) .
\end{split}
\ee
Here we used \eqref{eq-E} in the last step. 

Similarly, we have
\be \label{eq-f-b}
\begin{split}
\sin \bar \theta \,  g (\bar \nu_{_B} , \bar n_{_B} ) 
= & \ \frac12  \cos \bar \theta \,   \left( - b_i b_j  h_{i j} + a_i a_j  h_{ i j} \right)
+  \cos \bar \theta  - \cos \theta +  O ( r_{_P}^{-2 p} ) .
\end{split}
\ee
Therefore, by \eqref{eq-f-a} and \eqref{eq-f-b}, we conclude 
\be \label{eq-basic-angle-convex}
\begin{split}
 g (\bar \nu_{_A} , \bar n_{_A} )  +   g (\bar \nu_{_B} , \bar n_{_B} ) 
=  2  \left[  \frac{1}{ \sin \bar \theta}   ( \cos \bar \theta  - \cos \theta )  +  \frac{1}{ \sin \bar \theta}  O ( r_{_P}^{-2 p} ) \right] .
\end{split}
\ee

For later use, we note the relation between $\theta - \bar \theta$ and $ \alpha - \bar \alpha$ in this case.
Here $ \alpha $ is the dihedral angle of $P$ at $ F_{_A} \cap F_{_B}$ with respect to $g$.
Since $g$ is a metric continuously defined in a neighborhood of $\p P$, $ \alpha $ also satisfies 
\be \label{eq-alpha-theta}
 0 < \alpha < \pi  \ \  \text{and} \ \ \alpha + \theta = \pi . 
\ee
As a result,
\be \label{eq-alpha-theta-minus}
\theta - \bar \theta = \bar \alpha - \alpha . 
\ee

\vh

\noindent {\bf Case 2}. $ P$ is non-convex at $F_{_A} \cap F_{_B}$. This means $ \pi < \bar \alpha < 2 \pi $.

\begin{figure}[h]
\centering
\includegraphics[scale=1]{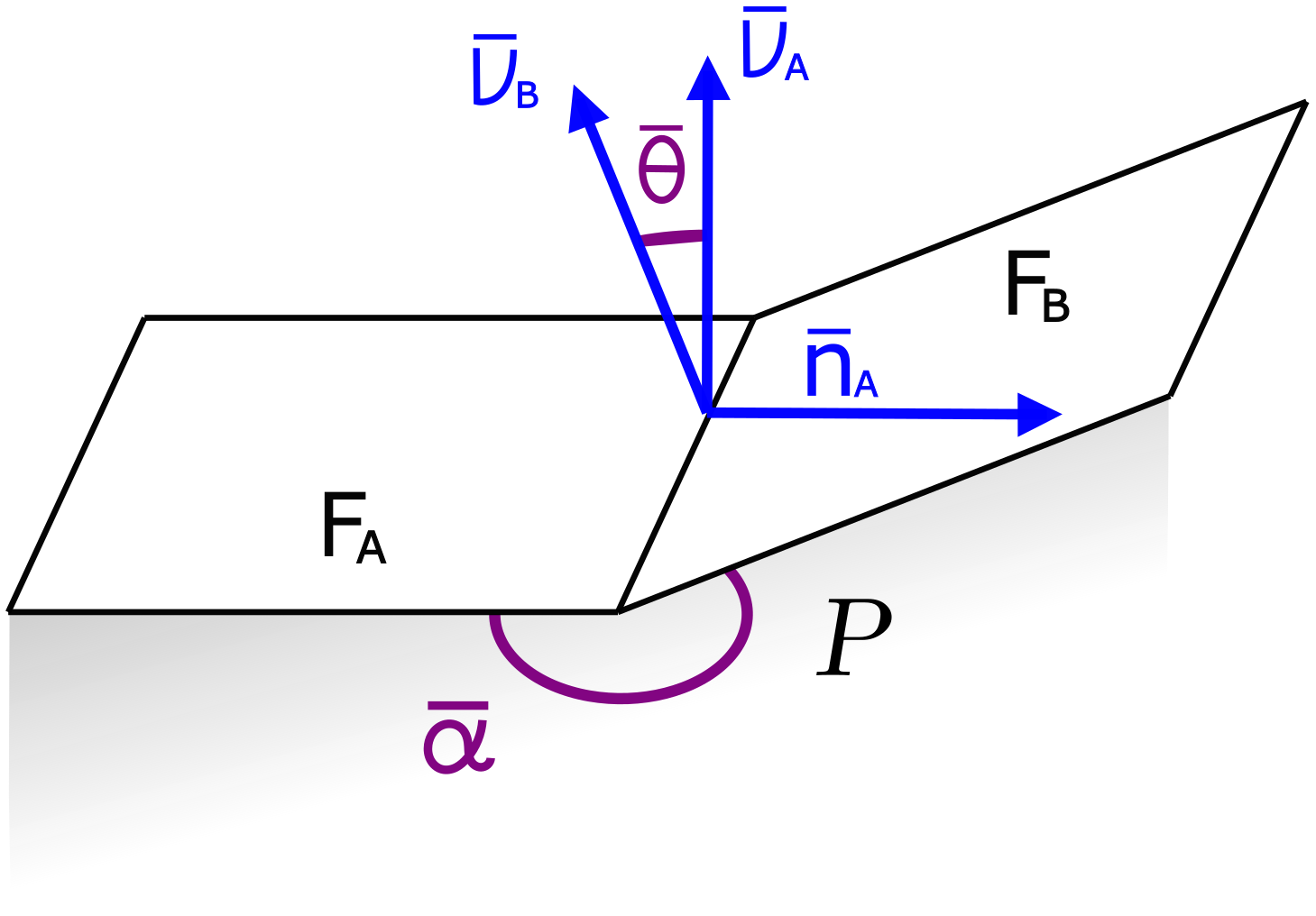}
\caption{$P$ is concave at $ F_{_A} \cap F_{_B}$.}
\label{fig-concave}
\end{figure}

In this case,
\be
\bar \alpha = \pi + \bar \theta \ \ \text{and} \ 
- (\sin \bar \theta) \, \bar n_{_A} = - (\cos \bar \theta) \,  \bar \nu_{_A} + \bar \nu_{_B} .
\ee
By similar calculations, we have
\be \label{eq-basic-angle-non-convex}
\begin{split}
 - g (\bar \nu_{_A} , \bar n_{_A} )  -   g (\bar \nu_{_B} , \bar n_{_B} ) 
=  2  \left[  \frac{1}{ \sin \bar \theta}   ( \cos \bar  \theta  - \cos \theta )  +  \frac{1}{ \sin \bar \theta}  O ( r_{_P}^{-2 p} ) \right] .
\end{split}
\ee
Also,  $ \alpha = \pi + \theta $ in this case. 
Consequently, 
\be \label{eq-alpha-theta-minus-1}
\theta - \bar \theta = \alpha - \bar \alpha. 
\ee

\vh
 
In either case, we can replace the term 
$$  \frac{1}{ \sin \bar \theta}   ( \cos \bar  \theta  - \cos \theta )  $$
via  $ (\theta - \bar \theta)$.   
By definition, $ \theta $ and $ \bar \theta$ satisfy 
$ 0 < \theta , \, \bar \theta < \pi $. Thus, by Taylor's theorem, 
\be \label{eq-mean-value}
\begin{split}
\theta - \bar \theta = & \  - \frac{1}{ \sin \bar \theta } ( \cos \theta - \cos \bar \theta )
- \frac12  \frac{ \cos \xi}{ (\sin \xi)^3} ( \cos \theta - \cos \bar \theta )^2
\end{split}
\ee
for some $ \xi $ between $ \theta $ and $ \bar \theta $. 
Combined with \eqref{eq-E}, this shows
\be \label{eq-mean-value-1}
\begin{split}
 \frac{1}{ \sin \bar \theta } ( \cos \theta - \cos \bar \theta ) = 
\bar \theta  -  \theta +  \frac{ 1 }{ (\sin \xi)^3}  O(r_{_P}^{-2 p} ) . 
\end{split}
\ee
Therefore, 
\begin{itemize}
\item[i)] if $ P$ is convex at $F_{_A} \cap F_{_B} $, 
by \eqref{eq-basic-angle-convex},  \eqref{eq-mean-value-1} and \eqref{eq-alpha-theta-minus}, 
\bee 
\begin{split}
& \  g (\bar \nu_{_A} , \bar n_{_A} )  +   g (\bar \nu_{_B} , \bar n_{_B} ) \\
=  & \  2  ( \theta - \bar \theta )  +  \frac{ 1 }{ ( \sin \xi )^3  }  O ( r_{_P}^{-2 p} ) +  \frac{1}{ \sin \bar \theta} O ( r_{_P}^{-2 p} )   \\
=  & \  2  ( \bar \alpha  - \alpha )  +  \frac{ 1 }{ ( \sin \xi )^3  }  O ( r_{_P}^{-2 p} ) +  \frac{1}{ \sin \bar \theta} O ( r_{_P}^{-2 p} )  ;
\end{split}
\eee
\item[ii)] if $ P$ is non-convex at $F_{_A} \cap F_{_B} $, 
by \eqref{eq-basic-angle-non-convex},  \eqref{eq-mean-value-1} and \eqref{eq-alpha-theta-minus-1}, 
\bee 
\begin{split}
& \  g (\bar \nu_{_A} , \bar n_{_A} )  +   g (\bar \nu_{_B} , \bar n_{_B} ) \\
=  & \  2  ( \bar \theta - \theta )  + \frac{ 1 }{ ( \sin \xi )^3  }  O ( r_{_P}^{-2 p} ) +  \frac{1}{ \sin \bar \theta} O ( r_{_P}^{-2 p} )  \\
=  & \  2  ( \bar \alpha - \alpha )  +  \frac{ 1 }{ ( \sin \xi )^3  }  O ( r_{_P}^{-2 p} ) +  \frac{1}{ \sin \bar \theta} O ( r_{_P}^{-2 p} )  .
\end{split}
\eee
\end{itemize}
Thus, regardless of the convexity of $P$ at $ F_{_A} \cap F_{_B}$, 
we always have
\be
 g (\bar \nu_{_A} , \bar n_{_A} )  +   g (\bar \nu_{_B} , \bar n_{_B} )  = 
 2  ( \bar \alpha - \alpha )  +  \frac{ 1 }{ ( \sin \xi )^3  }  O ( r_{_P}^{-2 p} ) +  \frac{1}{ \sin \bar \theta} O ( r_{_P}^{-2 p} ) .
\ee

\vh

To proceed, we impose an angle assumption 
\be
\sin \bar \theta \ge c ,
\ee
where $ c \in (0, 1)$ is a constant independent on $P$. This together with \eqref{eq-E} implies, 
for sufficiently large $r_{_P}$, 
$$
\sin \theta \ge \frac12 c .
$$
As a result, $ \sin \xi \ge \frac12 c$, and 
\be
 g (\bar \nu_{_A} , \bar n_{_A} )  +   g (\bar \nu_{_B} , \bar n_{_B} )  = 
 2  ( \bar \alpha - \alpha )  +  c^{-3} O ( r_{_P}^{-2 p} ) .
\ee
Moreover, by \eqref{eq-alpha-theta-minus}, \eqref{eq-alpha-theta-minus-1} and \eqref{eq-mean-value}, 
\be \label{eq-alpha-d}
| \bar \alpha - \alpha | = | \theta - \bar \theta | = c^{-3} O ( r_{_P}^{-p} ) . 
\ee

Returning to \eqref{eq-edge-contribution}, we  have
\be \label{eq-edge-contribution-1}
\begin{split}
& \ \int_{F_{_A} \cap F_{_B} } g (\bar \nu_{_A} , \bar n_{_A} )  +   g (\bar \nu_{_B} , \bar n_{_B} )  \, d \bar \mu \\
=  & \ 2  \int_{ F_{_A} \cap F_{_B} } ( \bar \alpha - \alpha ) \, d \bar \mu  
 +  | F_{_A} \cap F_{_B} |_{\bar \gamma} \, c^{-3} O ( r_{_P}^{-2 p} ) \\
= & \  2  \int_{ F_{_A} \cap F_{_B} } ( \bar \alpha - \alpha ) \, d  \mu  
 +  | F_{_A} \cap F_{_B} |_{\bar \gamma} \, c^{-3} O ( r_{_P}^{-2 p} ).
\end{split}
\ee
Here $ | {F}_{_A} \cap {F}_{_B} |_{\bar \gamma} $ is the $(n-2)$-dimensional volume of the edge
$F_{_A} \cap F_{_B}$ in $ (\R^n, \bar g)$, and $ d \mu $ is the volume element with respect to the metric 
induced from $g$. 

Combining \eqref{eq-int-H} -- \eqref{eq-edge-contribution}
and \eqref{eq-edge-contribution-1}, we obtain the following proposition. 

\begin{prop} \label{prop-main}
Let $ c \in (0,1)$ be a fixed constant. 
Suppose the polyhedron $P$ satisfies 
\be \label{eq-angle-condition}
 \sin \bar \theta \ge c 
\ee
along each edge of $P$.
If $ r_{_P}$ is sufficiently large, then 
\be 
\begin{split}
& \ \int_{\p P }  ( g_{ij,i} - g_{ii,j} ) \bar \nu^j \, d \bar \sigma \\
= & \  - 2  \int_{\mathcal{F}} H \, d \sigma +
  2  \int_{ \mathcal{E} } ( \alpha - \bar \alpha ) \, d  \mu  \\
& \  +   c^{-3} | \mathcal{E} |_{\bar \gamma} \, O ( r_{_P}^{-2 p} ) + 
 |\F |_{\bar \gamma }  \, O ( r_{_P}^{- 2 p - 1} ) + | \F |_{\bar \gamma} \, O ( r_{_P}^{-2 p - 1} ).
\end{split}
\ee
Here $ \F$ and $ \E$ are the union of all the faces and edges of $ P$, respectively. 
\end{prop}

Theorem \ref{thm-main} now follows from Proposition \ref{prop-main}. 
Take $P = P_k$, an element  in $\{ P_k \}$.
Since 
$$ | \sin \bar \alpha | = \sin \bar \theta , $$ 
condition d) is equivalent to $\sin \bar \theta \ge c$.
Therefore, by conditions a), b), c) and Proposition \ref{prop-main}, 
\be \label{eq-main-f}
\begin{split}
& \ \int_{\p P_k }  ( g_{ij,i} - g_{ii,j} ) \bar \nu^j \, d \bar \sigma \\
= & \  - 2  \int_{\F (\p P_k) } H \, d \sigma +  2  \int_{ \E (\p P_k)  } ( \alpha - \bar \alpha ) \, d  \mu  
+ o ( 1), \ \text{as} \ k \to \infty. 
\end{split}
\ee
Here we also used the decay condition $ p > \frac{n-2}{2}$.
Equation \eqref{eq-main} follows from \eqref{eq-mass-def-p} and \eqref{eq-main-f}.

\section{Integration of a lower dimensional mass-related quantity} \label{sec-cube}

We next consider the case in which $\lbrace P_k\rbrace$ is a sequence of large coordinate cubes. 
Cubes have a feature that, when sliced by hyperplanes parallel to a face, 
the resulting sections are $(n-1)$-dimensional large cubes as well. 
The following formula was derived in \cite[Equation (6)]{Miao19}.

\begin{thm}[\cite{Miao19}]
Let $(M^3, g)$ be an asymptotically flat $3$-manifold. 
Let $ C^3_L $ denote a large coordinate cube centered at the coordinate origin, 
with coordinate side length $2L$.
For each $ k=1, 2, 3$ and  each $ t \in [- L, L]$, 
let $S_t^{(k)}$ be the curve given by the intersection between $\partial C^3_L$ 
and the coordinate plane  $\lbrace x_k=t\rbrace$.
Then the mass of $(M^3,g)$ satisfies 
\begin{equation}\label{eq-GB}
\m (g) = \frac{1}{8\pi}\sum_{k=1}^3 \int_{-L}^L \m^{(2)}_k  ( t, L) \, dt + o (1) , \ \text{as} \ L \to \infty,
\end{equation}
where 
\be
\m^{(2)}_k (t, L) =  2\pi -\int_{S_t^{(k)}}\kappa^{(k)}ds -\beta^{(k)}_t  , 
\ee
$\kappa^{(k)}$ is the geodesic curvature of $S_t^{(k)}$ in $\lbrace x_k=t\rbrace$, 
and $\beta^{(k)}_t$ is the sum of the turning angle of $S_t^{(k)}$ at its four vertices. 
\end{thm}

As noted in \cite{Miao19}, the quantity $ \m_k^{(2)} (t, L) $ 
can be interpreted as an angle defect of the surface delimited by $S_t^{(k)}$ in $\{ x_k = t \}$. 
In the setting of asymptotically conical surfaces, it is known that this angle defect defines 
the $2$-d ``mass" of such surfaces (see \cite{CL2019} and \cite{Chr} for instance).   

\vspace{.1cm}

Combined with the Gauss-Bonnet formula and the work of Stern \cite{Stern19},  \eqref{eq-GB}
can be used to explain the recent proof of the $3$-dimensional positive mass theorem in \cite{BKKS19}. 
Motived by this, below we establish a higher dimensional analog of \eqref{eq-GB}. 

\begin{thm}\label{thm-prop}
Let $(M^n, g)$ be an asymptotically flat manifold with dimension $ n \ge 4$.
Given any index $k \in \{ 1, \cdots, n \}$,  any large constant $L$, and any $ t \in [-L, L]$,  
there is a quantity $ \m_k^{(n-1)} (t, L)$, associated to the coordinate hyperplane 
$\{ x_k = t \}$, defined in \eqref{eq-def-mktl} below, such that 
the mass of $(M^n, g)$ satisfies 
\begin{equation}\label{eq-slice}
(n-1) \omega_{n-1}  \m (g)= \omega_{n-2}  \sum_{k=1}^n \int_{-L}^L \m^{(n-1)}_k ( t, L ) \, dt + o (1) , 
\ \text{as} \ L \to \infty. 
\end{equation}
\end{thm}

To explain the quantity $\m_k^{(n-1)}( t, L) $, we first introduce some notations. 
Given a  large constant $L$,  let $C^n_L$ denote the coordinate cube in $(M, g)$, centered 
at the coordinate origin, with side length $2L$. 
For each $ i \in \{ 1, \cdots, n \}$,  let 
$$F^{(i)}_+=\lbrace x\in \partial C^n_L \,|\, x^i=  L \rbrace \ \ \text{and} \ \
F^{(i)}_-=\lbrace x\in \partial C^n_L \,|\, x^i= - L \rbrace , $$
which represent the front and back $i$-th face of $\p C^n_L$, respectively. 
Let $H_{i}$ denote the mean curvature of $F^{(i)}_\pm$ with respect to the outward normal $ \nu_i$
in $(M, g)$.
As before, we use $\alpha $ to denote the dihedral angle 
along every edge of $\p C^n_L$ in $(M, g)$. 

For each $ t  \in [-L, L]$,  
let $S_t^{(k)}$ be the intersection of $\partial C^n_L$ with the coordinate hyperplane $\{ x_k = t \}$, i.e.
$$S_t^{(k)}=\partial C^n_L\cap \lbrace x_k=t\rbrace.$$
$S_t^{(k)}$ is the boundary of an $(n-1)$-dimensional  cube in $\{ x_k = t \}$. (See Figure \ref{fig-cube}.)

\begin{figure}[ht]
\centering
\includegraphics[scale=.5]{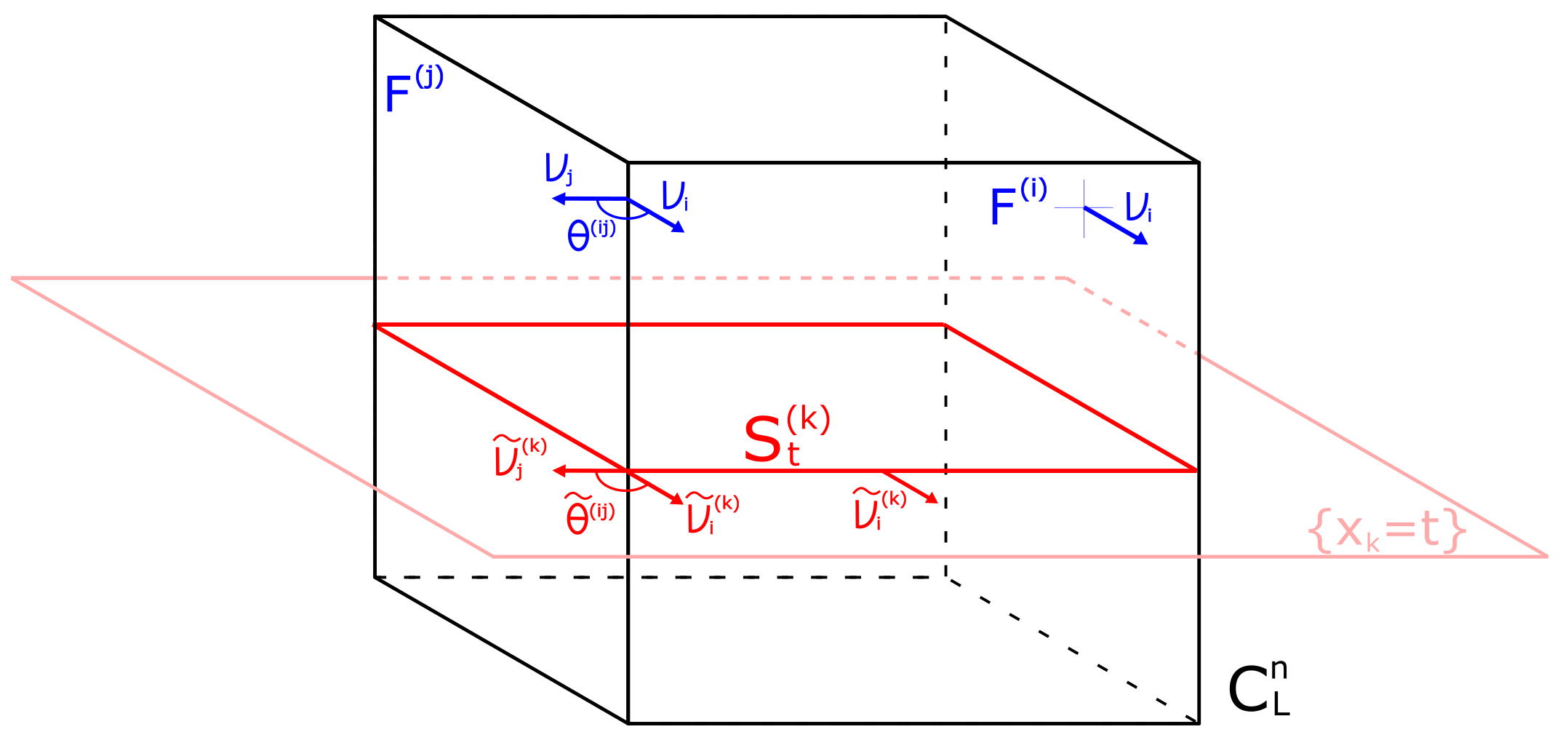}
\caption{$S^{(k)}_t$ is the boundary of an $(n-1)$-dimensional cube.}
\label{fig-cube}
\end{figure}

Within the hypersurface $\{ x_k = t \}$,  let  $\tilde{H}^{(k)}$ denote the mean curvature of 
each face of $S_t^{(k)}$ with respect to the outward normal $  \nu^{(k)}_i$.
Let $\tilde{\alpha}^{(k)}$ denote the dihedral angle along every edge of $ S^{(k)}_t$ in $\{ x_k = t \}$, with respect to $g$.

For  $ m \in \{ n-1, n-2 \}$, 
let $ d \sigma^m$, $ d \mu^{m-1}$, $ d \sigma^m_0$ and $ d \mu^{m-1}_0$  denote the relevant volume forms, 
induced from the metrics $g$ and $\bar g$, on an $m$-dimensional face and an $(m-1)$-dimensional edge of the corresponding cubes, respectively. 

Associated to each $ S^{(k)}_t$ in $\{ x_k = t \}$, define 
\be \label{eq-def-mktl}
\m_k^{(n-1)} (t, L) =  \frac{1}{ (n- 2)\omega_{n-2}  } \left(  - \int_{ \F ( S^{(k)}_t ) } \tilde{H}^{(k)} d\sigma^{n-2}  + \int_{\mathcal{E}(S_t^{(k)})}\( \tilde\alpha^{(k)} - \frac{\pi}{2}\right)d\mu^{n-3} \right) .
\ee

\begin{remark}
For each fixed $k$ and $t$, as a special case of Theorem \ref{thm-main}, we have
\be
\m ( g|_{  \{ x_k = t \} } ) = \lim_{ L \to \infty } \m_k^{(n-1)} (t, L) ,
\ee
where $ g |_{\{ x_k = t \} }$ denotes the induced metric on $\{ x_k = t \}$ from $g$.
In many cases, for instance if $g$ has a decay rate of $ p = n-2 $,  this limit will be zero
as $ \m ( g |_{ \{ x_k = t \} } ) = 0 $.
\end{remark}

To prove Theorem \ref{thm-prop}, we first relate the mean curvatures and the dihedral angles of 
$S^{(k)}_t$ in $\{ x_k = t \}$, $ k =1, \cdots, n $, to those of $ \p C^n_L $ in $(M, g)$.

\begin{lma}\label{lem}
For any $ i , k \in \{ 1, \cdots, n \}$ with $ i \ne k $, let  
$\tilde{H}^{(k)}_i$ be the mean curvature of $S_t^{(k)}\cap F^{(i)}_\pm$ in $\lbrace x_k=t\rbrace$
with respect to $  \nu^{(k)}_i$. Then 
\begin{equation} \label{eq-face} 
\sum_{k \in \, \{ 1, \cdots, n \} \setminus \{ i \} }
\tilde{H}^{(k)}_i=(n-2)H_{i}+ O(L^{-2p-1}), \ \text{as} \ L \to \infty. 
\end{equation}
Similarly, along each edge of $ S^{(k)}_t$ in $\{ x_k = t \}$, 
\begin{equation}\label{eq-angle}
\tilde{\alpha}^{(k)} =\alpha + O(L^{-2p}), \ \text{as} \ L \to \infty. 
\end{equation}
\end{lma}

\begin{proof}
It suffices to check \eqref{eq-face} and \eqref{eq-angle} on  $  F^{(i)}_+$ and along 
$ F^{(i)}_+ \cap F^{(j)}_+ $, where $ j \ne i $.
For any $ k \ne i $, at $S_t^{(k)}\cap F^{(i)}_+$, we have $ \nu_{i} =  \p_i + O ( L^{-p} )$ and 
$ \nu^{(k)}_i =  \p_i + O ( L^{-p} )$. Thus, 
\be \label{eq-ki-normal}
\nu_{i}= {\nu}^{(k)}_i+O(L^{p}) . 
\ee  
By definition, we have
\begin{equation}
\begin{split}
\sum_{k \in \, \{ 1, \cdots, n \} \setminus \{ i \}  } \tilde{H}^{(k)}_i  & = 
 \sum_{ k \in \, \{ 1, \cdots, n \} \setminus \{ i \}  } \ 
\sum_{\alpha,\beta \in \{ 1, \cdots, n \} \setminus \{ i,k \} } \ 
(-1) g^{\alpha\beta}\langle \nabla_{\partial_\alpha}\partial_\beta, {\nu}^{(k)}_i \rangle \\
&=  \sum_{ k \in \, \{ 1, \cdots, n \} \setminus \{ i \}  } \ 
\sum_{\alpha  \in \{ 1, \cdots, n \} \setminus \{ i,k \} } \, 
(-1) g^{\alpha \alpha }\langle \nabla_{\partial_\alpha}\partial_\alpha, {\nu}^{(k)}_i \rangle + O(L^{-2p-1})  \\
&=  \sum_{\alpha \in \{ 1, \cdots, n \} \setminus \{ i \} } \ 
\sum_{ k \in \{ 1, \cdots, n \} \setminus \{ i, \alpha \} } \,  
(-1) g^{\alpha \alpha}\langle \nabla_{\partial_\alpha}\partial_\alpha, {\nu}^{(k)}_i \rangle + O(L^{-2p-1})   .
\end{split}
\ee
Applying  \eqref{eq-ki-normal}, we have 
\be
\begin{split}
\sum_{k \in \, \{ 1, \cdots, n \} \setminus \{ i \}  } \tilde{H}^{(k)}_i  &=  \sum_{\alpha \in \{ 1, \cdots, n \} \setminus \{ i \} }  \, 
(n-2) \, (-1) g^{\alpha \alpha }\langle \nabla_{\partial_\alpha}\partial_\alpha,\nu_{ i }\rangle+ O(L^{-2p-1}) \\
&=(n-2)H_i+ O(L^{-2p-1}) ,
\end{split}
\end{equation}
which proves \eqref{eq-face}.

To prove \eqref{eq-angle}, by \eqref{eq-alpha-theta}, it suffices to check the corresponding relation for 
$\theta^{(ij)} $ and $ \tilde \theta^{(ij)} $. Here 
$\theta^{(ij)}$ is the angle between $ \nu_i$ and $ \nu_{j}$ along  the edge $ F^{(i)}_+ \cap F^{(j)}_{+}  $, 
and $\tilde \theta^{(ij)} $ is the angle between $ \nu^{(k)}_i$ and $ \nu^{(k)}_j $ along 
$ F^{(i)}_+ \cap F^{(j)}_{+}  \cap S_t^{(k)}$
(see Figure \ref{fig-cube}).

Applying \eqref{eq-E} to $C^n_L$ along $ F^{(i)}_+ \cap F^{(j)}_{+}  $ and noticing $ \bar \theta = \frac{\pi}{2} $ in this case,
we have
\begin{equation} \label{eq-theta-h}
\cos(\theta^{(ij)} ) = - h_{ij} + O (L^{-2p}) .
\ee
The same reason applied to $\{ x_k = t \}$ also gives 
\begin{equation}
\begin{split}
\cos(\tilde \theta^{(ij)} ) = - h_{ij} + O (L^{-2p}) .
\end{split}
\ee
As a result, we have $  \theta^{(ij)} = \frac{\pi}{2} + O (L^{-p}) $ and 
\be
\begin{split}
\cos(\theta^{(ij)} ) = \cos(\tilde{\theta}^{(ij)})+O(L^{-2p}).
\end{split}
\end{equation}
These readily imply 
\be \label{eq-theta-tilde-theta}
\theta^{(ij)} = \tilde{\theta}^{(ij)} + O(L^{-2p}).
\ee
Equation \eqref{eq-angle} follows from \eqref{eq-theta-tilde-theta} and \eqref{eq-alpha-theta}.
\end{proof}

We now prove Theorem \ref{thm-prop}. 
In what follows, we let $F^{(i)}=F^{(i)}_+\cup F^{(i)}_-$, 
and let $E^{(ij)}= F^{(i)} \cap F^{(j)} $ for $ i \ne j $.

\begin{proof}[Proof of Theorem \ref{thm-prop}]
By  \eqref{eq-face}, we have
\begin{equation}\label{eq-int-H-2}
\begin{split}
& \ \int_{\F (\p C^n_L) } H d \sigma^{n-1} 
= \ \sum_i\int_{ F^{(i)} } H_i \, d \sigma^{n-1}  \\
= & \ \frac{1}{n-2}\sum_i\sum_{k \in \{ 1, \cdots, n \} \setminus \{ i \} }\int_{ F^{(i)} } \tilde{H}^{(k)}_i  
d\sigma^{n-1}_0 + \o \\
= & \ \frac{1}{n-2}\sum_k 
 \sum_{i \in \{ 1, \cdots, n \} \setminus \{k \} }
\int_{-L}^L \left\lbrace \int_{  F^{(i)}  \cap S_t^{(k)}} \tilde{H}^{(k)}_i  d\sigma^{n-2}_0 \right\rbrace dt+\o\\
= & \ \frac{1}{n-2}\sum_k\int_{-L}^L \left\lbrace \int_{S^{(k)}_t} \tilde{H}^{(k)} d \sigma^{n-2} \right \rbrace dt+\o.\\
\end{split}
\end{equation}
By \eqref{eq-angle}, we have 
\begin{equation}\label{eq-int-angle}
\begin{split}
& \ \int_{\E (\p C^n_L) }\(\alpha- \frac{\pi}{2} \right)d\mu^{n-2} 
=   \frac12 \sum_{i \neq j}\int_{E^{(ij)}}\( \tilde{\alpha}^{(k)} -\frac{\pi}{2}\right)d\mu^{n-2}_0 +\o \\
= & \ \frac{1}{2 (n-2)}  \sum_{ i \neq j } \, 
\sum_{ k \in \{ 1, \cdots, n \} \setminus \{  i,j \}   } \, 
\int_{-L}^L \left\lbrace 
\int_{E^{(ij)}\cap S_t^{(k)}}\( \tilde{\alpha}^{(k)}   -\frac{\pi}{2}\right)d\mu_0^{n-3}\right\rbrace dt +\o\\
= & \ \frac{1}{n-2} \sum_k\int_{-L}^L  \left\lbrace\int_{\mathcal{E}(S_t^{(k)})} 
\( \tilde{\alpha}^{(k)} -\frac{\pi}{2}\right)d\mu^{n-3}\right\rbrace dt +\o . 
\end{split}
\end{equation}
Taking $ L \to \infty$, we conclude from Theorem \ref{thm-main}, \eqref{eq-int-H-2} and \eqref{eq-int-angle} that 
\begin{equation}
\begin{split}
& \ (n-1)\omega_{n-1} \mathfrak{m} (g) \\
= & \ -\int_{ \F ( \partial C_L^n )} H d\sigma^{n-1} + \int_{\mathcal{E}(\partial C^n_L)}\( \alpha- \frac{\pi}{2}\right)d\mu^{n-2}
+ o(1) \\
= & \ \frac{1}{n-2}\sum_k \int_{-L}^L   \left\lbrace -\int_{ \F ( S^{(k)}_t ) } \tilde{H}^{(k)} d\sigma^{n-2}  + \int_{\mathcal{E}(S_t^{(k)})}\( \tilde\alpha^{(k)}  - \frac{\pi}{2}\right)d\mu^{n-3} \right\rbrace dt + o(1) \\
= & \   \omega_{n-2} \sum_k \int_{-L}^L   \m_k^{ (n-1)} (t, L)   \, dt + o(1) .
\end{split} 
\end{equation}
This completes the proof.
\end{proof}


\bigskip


\begin{thebibliography}{10}

\bibitem{ADM} 
R. Arnowitt; S. Deser, and C. W. Misner, {\sl Coordinate invariance and energy expressions in general relativity}, Phys. Rev., {\bf 122} (1961), no. 3, 997--1006.

\bibitem{Bartnik86} 
R. Bartnik, {\sl The mass of an asymptotically flat manifold}, Comm. Pure Appl. Math.  \textbf{39} (1986), no. 5, 661--693.

\bibitem{BKKS19}
H. Bray, D. Kazaras, M. Khuri and D. Stern, 
{\it Harmonic functions and the mass of $3$-dimensional asymptotically flat Riemannian manifolds}, 
arXiv:1911.06754.

\bibitem{CL2019} A. Carlotto and C. De Lellis, \textit{Min-max embedded geodesic lines in asymptotically conical surfaces}, J. Differential Geom.\textbf{ 112} (2019), no. 3, 411–445.

\bibitem{Chr} 
P. Chru\'sciel, \textit{ Lectures on energy in general relativity}, preprint. [manuscipt available on \href{http://homepage.univie.ac.at/piotr.chrusciel/teaching/Energy/Energy.pdf}{http://homepage.univie.ac.at/piotr.chrusciel/teaching/Energy/Energy.pdf}]

\bibitem{Chrusciel}
P. Chru\'sciel,  {\sl Boundary conditions at spatial infinity from a Hamiltonian point of view},
Topological Properties and Global Structure of Space-Time, Plenum Press, New York, (1986), 49--59.


\bibitem{Gromov14}
M. Gromov, {\sl Dirac and Plateau billiards in domains with corners,} Cent. Eur. J. Math. \textbf{12} (2014), 
no. 8, 1109--1156.

\bibitem{Gromov18a}
M. Gromov, {\sl A dozen problems, questions and conjectures about positive scalar curvature}, 
Foundations of mathematics and physics one century after Hilbert, Springer, Cham, 2018, 135--158.


\bibitem{Li17}
C. Li,  {\it A polyhedron comparison theorem for three-manifolds with positive scalar curvature},
Invent. Math.  \textbf{219} (2020), 1--37.

\bibitem{Li20}
C. Li, {\sl The dihedral rigidity conjecture for $n$-prisms}, arXiv:1907.03855.


\bibitem{Lo3}
J. Lohkamp, {\sl Skin structures in scalar curvature geometry}, arXiv:1512.08252.


\bibitem{Miao02}
P. Miao, {\sl On existence of static metric extensions in general relativity}, Commun. Math. Phys. \textbf{241} (2003), 
no. 1, 27--46.

\bibitem{Miao19}
P. Miao,  {\sl Measuring mass via coordinate cubes},
Commun. Math. Phys. \textbf{379} (2020), 773--783.

\bibitem{MT09}
P.  Miao and L.-F. Tam, {\sl On the volume functional of compact manifolds with boundary with constant scalar curvature},  
Calc. Var. Part. Diff. Eq. \textbf{36}  (2009), no. 2, 141--171. 

\bibitem{SchoenYau79} 
R. Schoen and S.-T.  Yau,
{\sl On the proof of the positive mass conjecture in general relativity},
Commun.  Math. Phys. \textbf{65} (1979), no. 1, 45--76.

\bibitem{SchoenYau17} 
R. Schoen and S.-T.  Yau,
{\sl Positive scalar curvature and minimal hypersurfaces singularities}, 2017, arXiv:1704.05490.

\bibitem{Stern19}
D. Stern, {\sl Scalar curvature and harmonic maps to $\mathbb{S}^1$}, to appear in J. Diff. Geom., arXiv:1908.09754.

\bibitem{Witten81} 
E. Witten, 
{\sl A new proof of the positive energy theorem},
Commun.  Math. Phys. \textbf{80}  (1981), no. 3, 381--402.

\end{thebibliography}
\end{document}